\documentclass[12pt, a4paper]{amsart}
\usepackage{amsmath,mathrsfs}
\usepackage{geometry,amsthm,graphics,tabularx,amssymb,
shapepar}
\usepackage{amscd}
\usepackage[usenames]{color}
\usepackage{tikz-cd}
\usepackage{graphicx}

\usepackage[all]{xy}

\newcommand{\End}{{\mathrm{End}}}

\newcommand{\Gal}{{\mathrm{Gal}}}

\newcommand{\Hom}{{\mathrm{Hom}}}

\renewcommand{\Im}{{\mathrm{Im}}}

\newcommand{\Ext}{\operatorname{Ext}}

\newcommand{\oH}{\operatorname{H}}

\newcommand{\be}{\begin {equation}}
\newcommand{\ee}{\end {equation}}
\newcommand{\bee}{\begin {equation*}}
\newcommand{\eee}{\end {equation*}}

\theoremstyle{Theorem}

\newtheorem{thm}{Theorem}[section]

\newtheorem{lemt}[thm]{Lemma}
\newtheorem{prpt}[thm]{Proposition}

\theoremstyle{Theorem}

\theoremstyle{Theorem}
\newtheorem{prp}{Proposition}[section]
\newtheorem{corp}[prp]{Corollary}

\newtheorem{thmp}[prp]{Theorem}

\theoremstyle{Plain}

\theoremstyle{Definition}

\begin{document}

\title[Commutative Locally Nash Groups]{Classification of Connected Commutative Locally Nash Groups}

\author[Y. Bao]{Yixin Bao}

\address{School of Sciences, Harbin Institute of Technology, Shenzhen, 518055, China}
\email{mabaoyixin1984@163.com}

\author [Y. Chen] {Yangyang Chen}

\address{School of Sciences, Jiangnan University, Wuxi, 214122, China}
\email{8202007345@jiangnan.edu.cn}

\author [W. Hu]{Weikai Hu}
\address{UFR IM2AG, Uuiversit\'{e} Grenoble Alpes, Saint Martin D'H\`{e}res, France}
\email{weikai.hu26@gmail.com}




\begin{abstract}
In this article, we classify connected commutative (locally) Nash groups, which is a continuation of our previous work on the classification of abelian Nash manifolds. Our results generalize the classification of the one-dimensional case by Madden-Stanton and the two-dimensional case by Baro-Vicente-Otero. Moreover, we determine the affineness and toroidal affinenese of a connected commutative Nash group.
\end{abstract}

 \maketitle


\section{Introduction}
A locally Nash group is a group which is simultaneously a locally Nash manifold such that all group operations are locally Nash maps. A Nash manifold is a locally Nash manifold whose atlas has finitely many Nash charts, while a Nash group is a locally Nash group which is also a Nash manifold. For details of real algebraic geometry, we refer the readers to \cite{BCR} and \cite{Sh}. Since one can define Schwartz functions on Nash manifolds, Nash groups provide us a convenient setting for the study of infinite-dimensional smooth representations, see \cite{AGKL}, \cite{AGS}, \cite{CS} and \cite{SZ}. It is well-known that the category of Nash groups is closely related to the category of real algebraic groups. Motivated by the interest in real algebraic geometry and representation theory, many notions and theorems in the realm of real algebraic groups have been extended to the setting of Nash groups, and numerous results on the structure of Nash groups haven been achieved in recent years, see for example \cite{Sun}, \cite{FS} and \cite{Can}. In \cite{MS} and \cite{BVO}, the authors classified one-dimensional and two-dimensional connected commutative locally Nash groups. In \cite{BC}, the authors classified abelian Nash manifolds. In this article, we will generalize these results and classify connected commutative locally Nash groups with arbitrary dimensions.

A Nash manifold is said to be affine if it is Nash diffeomorphic to a Nash submanifold of
$\mathbb R^n$, for some $n\geq 0$, while a Nash group is called affine if the underlying Nash manifold is affine. By the work \cite{HP1}, \cite{HP2} and \cite{FS}, we know that affine Nash groups are precisely the finite covers of real algebraic groups. Inspired by this observation, the algebraization of affine Nash groups was introduced in \cite{FS} and plays a key role in the study of the structures of affine Nash groups. To study the structure of locally Nash groups (which are not necessarily affine), we need to generalize the algebraization method. For a real algebraic group $\mathsf{G}$, we denote by $\mathsf{G}(\mathbb{R})$ the set of its real points and by $\mathsf{G}(\mathbb{R})^{0}$ the connected component of $\mathsf{G}(\mathbb{R})$ containing the identity. For convenience, we call $\mathsf{G}(\mathbb{R})^{0}$ the identity component of $\mathsf{G}$ in this article. It is obvious that $\mathsf{G}(\mathbb{R})^{0}$ has a natural Nash group structure. Moreover, we have the following characterization of simply connected commutative locally Nash groups in terms of real algebraic groups.

\begin{prp}\label{algebraization}
Every simply connected commutative locally Nash group $G$ is locally Nash equivalent to the universal covering
of $\mathsf{G}(\mathbb{R})^{0}$, for some real connected commutative algebraic group $\mathsf{G}$.
\end{prp}

Note that the universal covering of a Nash group is naturally a locally Nash group, but not necessary a Nash group, c.f. \cite[Propositions 2.15 and 3.4]{Sun}. Two locally Nash groups are said to be locally Nash equivalent if there exists a locally Nash isomorphism between them.

\begin{prp}\label{csccng}
The locally Nash equivalence classes of simply connected commutative locally Nash groups are one-to-one corresponding to the isogeny classes of real connected commutative algebraic groups.
\end{prp}

Recall that an isogeny is a surjective morphism of algebraic groups with finite kernel.
By these two propositions, it is necessary for us to classify real connected commutative algebraic groups. In Section \ref{rcag1}, we review Brion's work on real connected commutative algebraic groups (c.f. \cite{Br}) and give a complete classification of them in terms of marked polarizable lattices. To make our parametrization compatible with those of real abelian varieties in \cite{BC}, we introduce the category of filtered polarizable lattices in Section \ref{rcag2} and prove the equivalence between the category of real connected commutative algebraic groups and the category of filtered polarizable lattices. To describe isogeny classes of real connected commutative algebraic groups, we slightly modify the definition of filtered polarizable lattices and introduce the category of filtered polarizable rational spaces. A filtered polarizable rational space is a pair $(\mathfrak{g}_{0},\Lambda_{r})$ where
\begin{itemize}
    \item $\mathfrak{g}_{0}$ is a real finite-dimensional vector space with a two-step filtration
    $$\mathrm{Fil}^{v} \mathfrak{g}_{0} \subseteq \mathrm{Fil}^{a} \mathfrak{g}_{0} \subseteq \mathfrak{g}_{0};$$
    \item $\Lambda_{r}$ is a rational subspace in $\mathfrak{g}=\mathfrak{g}_{0}\otimes_{\mathbb{R}}\mathbb{C}$ stable under the natural $\mathrm{Gal}(\mathbb{C}/\mathbb{R})$-action on $\mathfrak{g}$
\end{itemize}
such that
\begin{itemize}
    \item $\mathrm{Fil}^{v} \mathfrak{g} \cap \Lambda_{r} = 0;$
    \item $(\mathrm{Fil}^{a} \mathfrak{g} \cap \Lambda_{r})\otimes_{\mathbb{Q}}\mathbb{C}= \mathrm{Fil}^{a} \mathfrak{g}/ \mathrm{Fil}^{v} \mathfrak{g};$
    \item $(\Lambda_{r}/(\mathrm{Fil}^{a} \mathfrak{g} \cap \Lambda_{r}))\otimes_{\mathbb{Q}} \mathbb{R} = \mathfrak{g} /\mathrm{Fil}^{a} \mathfrak{g};$
    \item there exists a positive definite symmetric form $S$ on $\mathfrak{g}_{0} /\mathrm{Fil}^{a} \mathfrak{g}_{0}$ satisfying $E(x_{1}+y_{1}\sqrt{-1},x_{2}+y_{2}\sqrt{-1}):= S(y_{1},x_{2})-S(x_{1},y_{2})$ takes rational values on $\Lambda_{r}/(\mathrm{Fil}^{a} \mathfrak{g} \cap \Lambda_{r})$.
\end{itemize}

Here $\mathrm{Fil}^{v} \mathfrak{g}$ and $\mathrm{Fil}^{a} \mathfrak{g}$ denote the complexification of
$\mathrm{Fil}^{v} \mathfrak{g}_0$ and $\mathrm{Fil}^{a} \mathfrak{g}_0$, respectively.
A morphism from a $(\mathfrak{g}_{0,1},\Lambda_{r,1})$ to
$(\mathfrak{g}_{0,2},\Lambda_{r,2})$
is a real linear map  $\varphi$ from $\mathfrak{g}_{0,1}$ to $\mathfrak{g}_{0,2}$ such that
\begin{itemize}
    \item $\varphi_{\mathbb{C}}(\Lambda_{r,1}) \subseteq \Lambda_{r,2}$;
    \item $\varphi(\mathrm{Fil}^{v}\mathfrak{g}_{0,1}) \subseteq \mathrm{Fil}^{v}\mathfrak{g}_{0,2}$;
    \item $\varphi(\mathrm{Fil}^{a}\mathfrak{g}_{0,1}) \subseteq \mathrm{Fil}^{a}\mathfrak{g}_{0,2}$.
\end{itemize}
The category of filtered polarizable rational spaces is an additive category. Now we could state the first main result of this article.

\begin{thmp}\label{main1}
The category of simply connected commutative locally Nash groups is equivalent to the category of filtered polarizable rational spaces.
\end{thmp}

Generally, a connected commutative locally Nash group is locally Nash equivalent to the quotient of its universal covering by a discrete lattice. For these groups, we introduce the triple $(\mathfrak{g}_{0},\Lambda_{r},\Gamma)$, where $(\mathfrak{g}_{0},\Lambda_{r})$ is a filtered polarizable rational space as defined above and $\Gamma$ is a discrete lattice in $\mathfrak{g}_{0}$. A morphism from $(\mathfrak{g}_{0,1},\Lambda_{r,1},\Gamma_{1})$ to $(\mathfrak{g}_{0,1},\Lambda_{r,2},\Gamma_{2})$ is a morphism $\varphi$ from $(\mathfrak{g}_{0,1},\Lambda_{r,1})$ to $(\mathfrak{g}_{0,2},\Lambda_{r,2})$ in the category of filtered polarizable rational spaces such that $\varphi(\Gamma_{1})\subseteq \Gamma_{2}$.
We denote by $\mathcal D$ the category of such triples.

\begin{thmp}\label{main2}
The category of connected commutative locally Nash groups is equivalent to the category $\mathcal{D}$.
\end{thmp}

By Theorem \ref{main2}, connected commutative locally Nash groups are characterized by triples in the category $\mathcal D$.
To determine whether the locally Nash group associated with the triple $(\mathfrak{g}_{0},\Lambda_{r},\Gamma)$ is Nash, we introduce a subspace in $\mathfrak{g}_{0}$ called the essentially non-compact affine subspace as follows. Write $\sigma$ for the conjugate map in $\mathrm{Gal}(\mathbb{C}/\mathbb{R})$. We define
$$\mathfrak{g}_{0}^{\mathrm{enca}}:=\mathrm{Fil}^{v}\mathfrak{g}_{0}\oplus(\sqrt{-1}(\Lambda_{r}^{\sigma=-1}\otimes_{\mathbb{Q}}\mathbb{R})\cap\mathrm{Fil}^{a}\mathfrak{g}_{0}),$$
and have the following proposition.

\begin{prp}
Let $G$ be the connected commutative locally Nash group associated with the triple
$(\mathfrak{g}_{0},\Lambda_{r},\Gamma)$. Then the group $G$ is Nash if and only if the lattice $\Gamma$ spans the real vector space $\mathfrak{g}_{0}/\mathfrak{g}_{0}^{\mathrm{enca}}$.
\end{prp}

Recall that an affine Nash group is said to be anti-linear if it has no quotient Nash group which is almost
linear and positive dimensional while a toroidal affine Nash group \cite{Can} is defined to be a connected anti-linear
affine Nash group. In Section \ref{scclng}, we will define the split additive dimension $d_{v}$, the split multiplicative dimension $d_{t,+}$ and the split twisted multiplicative dimension $d_{t,-}$ for a connected commutative locally Nash group.
Using these data, we have the following criterion for affineness and toroidal affineness of a connected commutative locally Nash group.

\begin{prp}
Let $G$ be the connected commutative locally Nash group associated with the triple $(\mathfrak{g}_{0},\Lambda_{r},\Gamma)$. Then the group $G$ is affine if and only if $\Lambda_{r}\cap\mathfrak{g}_{0}=\Gamma\otimes_{\mathbb{Z}}\mathbb{Q}$. The group $G$ is toroidal affine if and only if $G$ is affine and $d_{v}=d_{t,+}=d_{t,-}=0$.
\end{prp}

The article is organized as follows. In Sections \ref{rcag1} and \ref{rcag2}, we review the classification of real connected commutative algebraic groups and parameterize these algebraic groups in terms of marked polarizable lattices and filtered polarizable lattices. In Section \ref{scclng}, we prove Propositions \ref{algebraization} and \ref{csccng}, and classify simply connected commutative locally Nash groups in terms of filtered polarizable rational spaces. In Section \ref{cclng}, we give a complete classification of connected commutative locally Nash groups and of connected commutative Nash groups. As an application of these classifications, we determine the conditions for affineness and toroidal affineness of connected commutative Nash groups in Section \ref{affinetoroidal}. Finally in the last section, we show that our classification matches that of one-dimensional and two-dimensional cases in \cite{MS} and \cite{BVO}.

{\bf Acknowledgements.} The authors would like to thank Doctor Yuancao Zhang for his valuable suggestions on real algebraic geometry. Yixin Bao is supported by the NSFC (Grant No.11801117) and the Natural Science Foundation of Guangdong Province, China (Grant No.2018A030313268).

\section{Classification of Real Commutative Algebraic Groups I}\label{rcag1}

In this section, we review the classification of real connected commutative algebraic groups (c.f. \cite{Br}).

For any real connected commutative algebraic group $\mathsf{G}$, there is a unique maximal real connected affine subgroup
$\mathsf{G}_{\mathrm{aff}}$. It is the direct product of a real torus $\mathsf{T}$ and a vector group $\mathsf{V}$. The quotient group $\mathsf{A}=\mathsf{G}/\mathsf{G}_{\mathrm{aff}}$ is a real abelian variety. On the other hand, $\mathsf{G}$ has a unique minimal connected algebraic subgroup $\mathsf{G}_{\mathrm{ant}}$ such that $\mathsf{G}/\mathsf{G}_{\mathrm{ant}}$ is affine. Moreover, $\mathsf{G}_{\mathrm{ant}}$ is anti-affine. Recall that an algebraic group is called anti-affine if every global regular function is constant.

\begin{prp}\label{cagpp1}
The real connected commutative algebraic group $\mathsf{G}$ is canonically isomorphic to $\mathsf{G}/\mathsf{V} \times_{\mathsf{A}} \mathsf{G}/\mathsf{T}$.
\end{prp}

\begin{proof}
By the universal property of fiber product, we have the following exact sequence
$$\xymatrix{
  0 \ar[r] & \mathsf{T}\times \mathsf{V} \ar[r] & \mathsf{G}/\mathsf{V} \times_{\mathsf{A}} \mathsf{G}/\mathsf{T} \ar[r] & \mathsf{A} \ar[r] & 0   }.$$
The universal property of $\mathsf{G}/\mathsf{V} \times_{\mathsf{A}} \mathsf{G}/\mathsf{T}$ induces a morphism from
$\mathsf{G}$ to $\mathsf{G}/\mathsf{V} \times_{\mathsf{A}} \mathsf{G}/\mathsf{T}$. Thus we have the following commutative diagram:
\begin{equation*}
    \xymatrix{
       0 \ar[r] & \mathsf{T}\times \mathsf{V} \ar@{=}[d]\ar[r] & \mathsf{G} \ar[d] \ar[r] & \mathsf{A} \ar@{=}[d] \ar[r] & 0  \\
      0 \ar[r] & \mathsf{T}\times \mathsf{V}\ar[r] & \mathsf{G}/\mathsf{V} \times_{\mathsf{A}} \mathsf{G}/\mathsf{T} \ar[r] & \mathsf{A} \ar[r] & 0.  }
\end{equation*}
By five-lemma, we know the middle arrow is an isomorphism.
\end{proof}

To parametrize real connected commutative algebraic group, we introduce the category $\mathcal{C}'$ of marked polarizable lattices. A marked polarizable lattice is a $6$-tuple $(V_{a},\Lambda_{a}, V_{v}, \Lambda_{t} ,\phi_{v}, \phi_{t})$ where
\begin{itemize}
    \item $V_{a}$ is a real finite-dimensional vector space;
    \item $\Lambda_{a}$ is a full lattice in $V_{a}\otimes_{\mathbb{R}}\mathbb{C}$ stable under the natural $\Gal(\mathbb{C}/\mathbb{R})$-action;
    \item $V_{v}$ is another real finite-dimensional vector space;
    \item $\Lambda_{t}$ is a free $\mathbb{Z}$-module of finite rank with a $\Gal(\mathbb{C}/\mathbb{R})$-action;
    \item $\phi_{v}: V_{v} \to V^{\check{}}_{a}$ is $\mathbb{R}$-linear;
    \item $\phi_{t}: \Lambda_{t} \to (V_{a}\otimes_{\mathbb{R}}\mathbb{C})^{\check{}}/\Lambda^{\check{}}_{a}$ is $\mathbb{Z}[\Gal(\mathbb{C}/\mathbb{R})]$-linear
\end{itemize}
such that
\begin{itemize}
    \item there exists a positive definite symmetric form $S$ on $V_{a}$ satisfying $E(x_{1}+y_{1}\sqrt{-1},x_{2}+y_{2}\sqrt{-1}):= S(y_{1},x_{2})-S(x_{1},y_{2})$ takes integer values on $\Lambda_{a}$.
\end{itemize}
Here $V^{\check{}}_{a}$ is the dual space of $V_{a}$, $(V_{a}\otimes_{\mathbb{R}}\mathbb{C})^{\check{}}$ is the space of anti-linear maps on $V_{a}\otimes_{\mathbb{R}}\mathbb{C}$ and its submodule of maps whose imaginary parts take integral values on $\Lambda_{a}$ is naturally identified with $\Lambda^{\check{}}_{a}$. Notice that the pair $(V_{a},\Lambda_{a})$ is a so-called real polarizable lattice (see \cite{BC}). A morphism from a
$(V_{a,1},\Lambda_{a,1}, V_{v,1}, \Lambda_{t,1} ,\phi_{v,1}, \phi_{t,1})$ to
$(V_{a,2},\Lambda_{a,2}, V_{v,2}, \Lambda_{t,2} ,\phi_{v,2}, \phi_{t,2})$
is a triple $(\varphi_{a}, \varphi_{v}, \varphi_{t})$ where
\begin{itemize}
    \item $\varphi_{a}$ is a real linear map from $V_{a,1}$ to $V_{a,2}$;
    \item $\varphi_{v}$ is a real linear map from $V_{v,2}$ to $V_{v,1}$;
    \item $\varphi_{t}$ is a $\mathbb{Z}[\Gal(\mathbb{C}/\mathbb{R})]$-module homomorphism from $\Lambda_{t,2}$ to $\Lambda_{t,1}$
\end{itemize}
such that
\begin{itemize}
    \item $\varphi_{a,\mathbb{C}}(\Lambda_{a,1}) \subseteq \Lambda_{a,2}$,;
    \item $\phi_{v,1}\circ \varphi_{v}= \varphi^{\check{}}_{a}\circ \phi_{v,2}, ~ \phi_{t,1}\circ \varphi_{t}= \varphi^{\check{}}_{a,\mathbb{C}}\circ \phi_{t,2}$.
\end{itemize}
Here $\varphi_{a,\mathbb{C}}=\varphi_{a}\otimes_{\mathbb{R}}\mathbb{C}$ and $\varphi^{\check{}}_{a}$ is the natural linear map induced by $\varphi_{a}$.

Now we prove the equivalence between the category $\mathcal{C}$ of real connected commutative algebraic groups and the category of marked polarizable lattices. Given a real connected commutative algebraic group $\mathsf{G}$, we write
$$\mathsf{M}:=\mathsf{G}/\mathsf{V}, \quad \mathsf{N}:=\mathsf{G}/\mathsf{T},$$
where $\mathsf{T}$ and $\mathsf{V}$ are those groups as mentioned above. Recall the quotient group $\mathsf{A}=\mathsf{G}/\mathsf{G}_{\mathrm{aff}}$. We have two exact sequences:
\begin{gather*}
  \xymatrix{
    0 \ar[r] & \mathsf{T} \ar[r] & \mathsf{M} \ar[r] & \mathsf{A} \ar[r] & 0,   } \\
  \xymatrix{
    0 \ar[r] & \mathsf{V} \ar[r] & \mathsf{N} \ar[r] & \mathsf{A} \ar[r] & 0.   }
\end{gather*}
From the first exact sequence, we naturally induce a map $\phi_t$ from $\Hom(\mathsf{T}_{\mathbb{C}}, \mathsf{G}_{m,\mathbb{C}})$ to $\Ext (\mathsf{A}_{\mathbb{C}}, \mathsf{G}_{m,\mathbb{C}})$, where $\mathsf{G}_{m,\mathbb{C}}$ is the multiplicative group over complex numbers. From the second exact sequence, we naturally induced a map $\phi_v$ from $\Hom(\mathsf{V}, \mathsf{G}_{a})$ to $\Ext (\mathsf{A}, \mathsf{G}_{a})$, where $\mathsf{G}_{a}$ is the real additive group. By Theorem 1.3 of \cite{BC}, we associate the abelian variety $\mathsf{A}$ with a real polarizable lattice $(V_{a},\Lambda_{a})$. By \cite[Chapter 7, Theorem 7]{Ser} and \cite[Theorem 1]{Ros}, the group $\Ext (\mathsf{A}_{\mathbb{C}}, \mathsf{G}_{m,\mathbb{C}})$ is canonically isomorphic to $\oH^{1}(\mathsf{A}_{\mathbb{C}},\mathcal{O}_{\mathsf{A}_{\mathbb{C}}}^{*})$, while the group $\Ext (\mathsf{A}, \mathsf{G}_{a})$ is canonically isomorphic to $\oH^{1}(\mathsf{A},\mathcal{O}_{\mathsf{A}})$. It is routine to check that $\oH^{1}(\mathsf{A},\mathcal{O}_{\mathsf{A}})$ is isomorphic to $V^{\check{}}_{a}$ and $\oH^{1}(\mathsf{A}_{\mathbb{C}},\mathcal{O}_{\mathsf{A}_{\mathbb{C}}}^{*})$ is isomorphic to $(V_{a}\otimes_{\mathbb{R}}\mathbb{C})^{\check{}}/\Lambda^{\check{}}_{a}$. So in this way we obtain a marked polarizable lattice $$(V_{a},\Lambda_{a}, \Hom(\mathsf{V}, \mathsf{G}_{a}), \Hom(\mathsf{T}_{\mathbb{C}}, \mathsf{G}_{m,\mathbb{C}}), \phi_{v}, \phi_{t}).$$
A morphism between two algebraic groups naturally induces a morphism between two marked polarizable lattices.

\begin{thmp}\label{ccc}
The above construction induces an equivalence from the category $\mathcal{C}$ of real connected commutative algebraic groups to the category $\mathcal{C}'$ of marked polarizable lattices.
\end{thmp}

\begin{proof}
The affine case is well known and the abelian case is proved in \cite{BC}. For general cases, the extension classes of a real vector group $\mathsf{V}$ and a real abelian variety $\mathsf{A}$ are classified by the homomorphisms from the cyclic-free right $\End(\mathsf{V})$-module $\Hom(\mathsf{V},\mathsf{V})$ to $\Ext(\mathsf{A}, \mathsf{V})$. On the other side, we have the canonical map from $\Hom_{\mathbb{R}}(\Hom(\mathsf{V},\mathsf{G}_{a}), \Ext(\mathsf{A}, \mathsf{G}_{a}))$ to
$\Hom_{\End(\mathsf V)}(\Hom(\mathsf{V},\mathsf{V}), \Ext(\mathsf{A},\mathsf{V}))$, which sends a map $\varphi$ in the former linear space to a map $\varphi\otimes_{\mathbb{R}}\mathsf V$. This map is an isomorphism because $\mathsf{V}$ is a direct sum of copies of $\mathsf{G}_{a}$. So there is a natural isomorphism between $$\Hom_{\mathbb{R}}(\Hom(\mathsf{V},\mathsf{G}_{a}), \Ext(\mathsf{A}, \mathsf{G}_{a}))$$ and $\Ext(\mathsf{A}, \mathsf{V})$. Similarly, there is a natural isomorphism between
$$\Hom_{\mathbb{Z[\Gal(\mathbb{C}/\mathbb{R})]}}(\Hom(\mathsf{T}_{\mathbb{C}}, \mathsf{G}_{m,\mathbb{C}}), \Ext(\mathsf{A}_{\mathbb{C}},\mathsf{G}_{m,\mathbb{C}}))$$
and $\Ext(\mathsf{A}, \mathsf{T})$. As in the proof of Proposition \ref{cagpp1}, the real connected commutative algebraic group is uniquely determined by these two extensions. This shows that the functor is essentially surjective.

For any commutative diagram in the categories of commutative algebraic groups
\begin{equation*}
    \xymatrix{
       0  \ar[r] & \mathsf{G}_{\mathrm{aff},1} \ar[d]_{0} \ar[r] & \mathsf{G}_{1} \ar[d] \ar[r] & \mathsf{A}_{1} \ar[d]_{0} \ar[r] & 0  \\
      0 \ar[r] & \mathsf{G}_{\mathrm{aff},2} \ar[r] & \mathsf{G}_{2} \ar[r] & \mathsf{A}_{2} \ar[r] & 0,  }
\end{equation*}
we have an induced map from $\mathsf{A}_{1}$ to $\mathsf{G}_{\mathrm{aff,2}}$ by snake lemma. The induced map is zero since $\Hom(\mathsf{A}_{1}, \mathsf{G}_{\mathrm{aff,2}}) = 0$. So the middle vertical map in above commutative diagram must be zero. Thus the functor is faithful.

To show that this functor is full. By the knowledge of algebraic groups and \cite{BC}, the map $\varphi_{a}$ and $\varphi_{v}$ in the tripe are induced from a morphism between two real abelian varieties $\mathsf{A}_{1}$ and $\mathsf{A}_{2}$, and a morphism between two vector groups $\mathsf{V}_{1}$ and $\mathsf{V}_{2}$, respectively. Without causing confusion, we also denote these two morphisms by $\varphi_{a}$ and $\varphi_{v}$, respectively. We have the following commutative diagram:
\begin{equation*}
  \xymatrix{
    \Hom(\mathsf{V}_2, \mathsf{G}_{a}) \ar[d]_{\varphi_v} \ar[r]^{\phi_v,2} & \Ext(\mathsf{A}_2, \mathsf{G}_{a}) \ar[d]^{\varphi^{\check{}}_{a}} \\
    \Hom(\mathsf{V}_1, \mathsf{G}_{a}) \ar[r]^{\phi_v,1} & \Ext(\mathsf{A}_1, \mathsf{G}_{a}).   }
\end{equation*}
Let $\mathsf{N}_{1}$, $\mathsf{N}_{2}$ and $\mathsf{N}_{3}$ be the extension groups determined by $\phi_{v,1}$, $\phi_{v,2}$ and $\varphi^{\check{}}_{a}\circ\phi_{v,2}$, respectively. The commutative diagram could be decomposed into two commutative diagrams:
\begin{equation*}
  \xymatrix{
    \Hom(\mathsf{V}_2, \mathsf{G}_{a}) \ar[d]_{\varphi_v} \ar[r]^{\varphi^{\check{}}_{a}\circ\phi_v,2} & \Ext(\mathsf{A}_1, \mathsf{G}_{a}) \ar@{=}[d] \\
    \Hom(\mathsf{V}_1, \mathsf{G}_{a}) \ar[r]^{\phi_v,1} & \Ext(\mathsf{A}_1, \mathsf{G}_{a}),   } \quad \xymatrix{
    \Hom(\mathsf{V}_2, \mathsf{G}_{a}) \ar@{=}[d] \ar[r]^{\phi_v,2} & \Ext(\mathsf{A}_2, \mathsf{G}_{a}) \ar[d]^{\varphi^{\check{}}_{a}} \\
    \Hom(\mathsf{V}_2, \mathsf{G}_{a}) \ar[r]^{\varphi^{\check{}}_{a}\circ\phi_v,2} & \Ext(\mathsf{A}_1, \mathsf{G}_{a}).   }
\end{equation*}
Hence we have the following commutative diagrams:
\begin{gather*}
  \xymatrix{
       0 \ar[r] & \mathsf{V}_1 \ar[d]_{\varphi_v} \ar[r] & \mathsf{N}_1 \ar[d] \ar[r] & \mathsf{A}_1 \ar@{=}[d] \ar[r] & 0  \\
    0 \ar[r] & \mathsf{V}_2 \ar[r] & \mathsf{N}_3 \ar[r] & \mathsf{A}_1 \ar[r] & 0,   } \\
  \xymatrix{
       0 \ar[r] & \mathsf{V}_2 \ar@{=}[d] \ar[r] & \mathsf{N}_3 \ar[d] \ar[r] & \mathsf{A}_1 \ar[d]^{\varphi_a} \ar[r] & 0  \\
    0 \ar[r] & \mathsf{V}_2 \ar[r] & \mathsf{N}_2 \ar[r] & \mathsf{A}_2 \ar[r] & 0.   }
\end{gather*}
Thus we obtain a morphism from $\mathsf{N}_{1}$ to $\mathsf{N}_{2}$. Similar result holds for the torus case. Then the fullness of the functor follows from Proposition \ref{cagpp1}.
\end{proof}

We use the subscript $"\mathrm{tor}"$ to denote the torsion part of a finitely generated abelian group.

\begin{corp}\label{antparametrization}
For a given real connected commutative algebraic group $\mathsf{G}$, if the associated marked polarizable lattice is $(V_{a},\Lambda_{a},V_{v},\Lambda_{t},\phi_{v},\phi_{t})$,
then the marked polarizable lattice associated with $\mathsf{G}_{\mathrm{ant}}$ is
    $$(V_{a},\Lambda_{a},V_{v}/\ker{\phi_{v}}, (\Lambda_{t}/\ker{\phi_{t}})/(\Lambda_{t}/\ker{\phi_{t}})_{\mathrm{tor}},\phi_{v}, \phi_{t}).$$
\end{corp}

\begin{proof}
It is strightforward.
\end{proof}

\section{Classification of Real Commutative Algebraic Groups II}\label{rcag2}
In this section we classify real connected commutative algebraic groups using another linear data called filtered polarizable lattices.

We define the category $\mathcal{C}''$ of filtered polarizable lattices as follows. A filtered polarizable lattice is a pair
$(\mathfrak{g}_{0},\Lambda)$ where
\begin{itemize}
    \item $\mathfrak{g}_{0}$ is a real finite-dimensional vector space with a two-step filtration
    $$\mathrm{Fil}^{v} \mathfrak{g}_{0} \subseteq \mathrm{Fil}^{a} \mathfrak{g}_{0} \subseteq \mathfrak{g}_{0};$$
    \item $\Lambda$ is a discrete lattice in $\mathfrak{g}=\mathfrak{g}_{0}\otimes_{\mathbb{R}}\mathbb{C}$ stable under the natural $\mathrm{Gal}(\mathbb{C}/\mathbb{R})$-action
\end{itemize}
such that
\begin{itemize}
    \item $\mathrm{Fil}^{v} \mathfrak{g} \cap \Lambda = 0;$
    \item $(\mathrm{Fil}^{a} \mathfrak{g} \cap \Lambda)\otimes_{\mathbb{Z}}\mathbb{C}= \mathrm{Fil}^{a} \mathfrak{g}/ \mathrm{Fil}^{v} \mathfrak{g};$
    \item $\Lambda/(\mathrm{Fil}^{a} \mathfrak{g} \cap \Lambda)$ is a full lattice in $\mathfrak{g} /\mathrm{Fil}^{a} \mathfrak{g} $;
    \item there exists a positive definite symmetric form $S$ on $\mathfrak{g}_{0} /\mathrm{Fil}^{a} \mathfrak{g}_{0}$ satisfying $E(x_{1}+y_{1}\sqrt{-1},x_{2}+y_{2}\sqrt{-1}):= S(y_{1},x_{2})-S(x_{1},y_{2})$ takes integer values on $\Lambda/(\mathrm{Fil}^{a} \mathfrak{g} \cap \Lambda)$.
\end{itemize}
Here $\mathrm{Fil}^{v} \mathfrak{g}$ and $\mathrm{Fil}^{a} \mathfrak{g}$ are the complexified spaces of $\mathrm{Fil}^{v} \mathfrak{g}_{0}$ and $\mathrm{Fil}^{a} \mathfrak{g}_{0}$, respectively. A morphism from
$(\mathfrak{g}_{0,1},\Lambda_{1})$ to
$(\mathfrak{g}_{0,2},\Lambda_{2})$
is a real linear map  $\varphi$ from $\mathfrak{g}_{0,1}$ to $\mathfrak{g}_{0,2}$ such that
\begin{itemize}
    \item $\varphi_{\mathbb{C}}(\Lambda_{1}) \subseteq \Lambda_{2}$;
    \item $\varphi(\mathrm{Fil}^{v}\mathfrak{g}_{0,1}) \subseteq \mathrm{Fil}^{v}\mathfrak{g}_{0,2}$.
    \item $\varphi(\mathrm{Fil}^{a}\mathfrak{g}_{0,1}) \subseteq \mathrm{Fil}^{a}\mathfrak{g}_{0,2}$.
\end{itemize}
This is an additive category.

Now we try to show the equivalence between the category $\mathcal{C}$ of real connected commutative algebraic groups and the category $\mathcal{C}''$ of filtered polarizable lattices. Given a real connected commutative algebraic group $\mathsf{G}$, we have its subgroups $\mathsf{G}_{\mathrm{aff}}$ and $\mathsf{V}$ as in the previous section. Let $\mathfrak{g}_{0}$,
$\mathrm{Fil}^{v} \mathfrak{g}_{0}$ and $\mathrm{Fil}^{a} \mathfrak{g}_{0}$ be the real Lie algebras of $\mathsf{G}$,
$\mathsf{V}$ and $\mathsf{G}_{\mathrm{aff}}$, respectively. We view $\mathsf{G}_{\mathbb{C}}$ as a complex Lie group and let $\Lambda$ be the fundamental group of $\mathsf{G}_{\mathbb{C}}$. Then $(\mathfrak{g}_{0}, \Lambda)$ is a filtered polarizable lattice. In this way we get an additive functor from $\mathcal{C}$ to $\mathcal{C}''$.

On the other hand, we construct an additive functor from $\mathcal{C}''$ to $\mathcal{C}'$. Given a filtered polarizable lattice $(\mathfrak{g}_{0}, \Lambda)$, let $V_{a}$ denote the real vector space $\mathfrak{g}_{0}/\mathrm{Fil}^{a} \mathfrak{g}_{0}$, $\Lambda_{a}$ denote the free $\mathbb{Z}$-module $\Lambda / (\mathrm{Fil}^{a} \mathfrak{g}\cap \Lambda)$, $V_{v}$ denote $(\mathrm{Fil}^{v} \mathfrak{g}_{0})^{\check{}}$, the dual space of $\mathrm{Fil}^{v} \mathfrak{g}_{0}$ and $\Lambda_{t}$ denote the free $\mathbb{Z}$-module $\sqrt{-1}(\mathrm{Fil}^{a} \mathfrak{g}\cap \Lambda)^{\check{}}$. It remains for us to construct the maps $\phi_{v}$ and $\phi_{t}$. By the definition of filtered polarizable lattices we have the following direct sum decomposition
$$\mathrm{Fil}^{a} \mathfrak{g} = \mathrm{Fil}^{v} \mathfrak{g}\oplus (\Lambda_{t}\otimes_{\mathbb{Z}}\mathbb{C}).$$
Denote the latter factor by $V_{t}$. We get two connected commutative complex Lie groups $$\mathfrak{n}:=\mathfrak{g}/(V_{t}+\Lambda), \quad \mathfrak{m}:=\mathfrak{g}/(\mathrm{Fil}^{v} \mathfrak{g}+\Lambda).$$The first group has a closed subgroup $\mathrm{Fil}^{v} \mathfrak{g}$. So it fits the following exact sequence:
$$0 \to \mathrm{Fil}^{v} \mathfrak{g} \to \mathfrak{n} \to \mathrm{A}_{\mathbb{C}} \to 0,$$
where $\mathrm{A}_{\mathbb{C}}$ denote the complex Lie group $(V_{a}\otimes_{\mathbb{R}}\mathbb{C})/\Lambda_{a}$. For any complex linear map from $\mathrm{Fil}^{v} \mathfrak{g}$ to $\mathbb{C}$, the graph map embedded $\mathrm{Fil}^{v} \mathfrak{g}$ into $\mathfrak{n}\times \mathbb{C}$ as a closed subgroup. Then we get the following exact sequence:
$$0 \to \mathbb{C} \to (\mathfrak{n} \times \mathbb{C})/\mathrm{Fil}^{v} \mathfrak{g}\to \mathrm{A}_{\mathbb{C}}\to 0.$$
So we get a principal $\mathbb{C}$-bundle on $\mathrm{A}_{\mathbb{C}}$. It determines an element in $\oH^{1}(\mathrm{A}_{\mathbb{C}}, \mathscr{O}^{an}_{\mathrm{A}_{\mathbb{C}}})$ which is canonically isomorphic to $(\mathfrak{g}/\mathrm{Fil}^{a} \mathfrak{g})^{\check{}}$, namely, the space of anti-linear maps on $\mathfrak{g}/\mathrm{Fil}^{a} \mathfrak{g}$. So we get a map from $\Hom_{\mathbb{C}}(\mathrm{Fil}^{v} \mathfrak{g},\mathbb{C})$ to $(\mathfrak{g}/\mathrm{Fil}^{a} \mathfrak{g})^{\check{}}$. It is routine to check that the map is $\mathbb{C}$-linear and preservers the $\mathrm{Gal}(\mathbb{C}/\mathbb{R})$-action. Taking the $\mathrm{Gal}(\mathbb{C}/\mathbb{R})$-fixed part, we get the map $\phi_{v}$ from $V_{v}$ to $V_{a}^{\check{}}$.

Similarly, we consider the complex Lie group $\mathfrak{m}$. So it fits the following exact sequence:
$$0 \to V_{t}/(\mathrm{Fil}^{a} \mathfrak{g}\cap \Lambda) \to \mathfrak{m} \to \mathrm{A}_{\mathbb{C}} \to 0.$$
For any homomorphism of Lie groups from $V_{t}/(\mathrm{Fil}^{a} \mathfrak{g}\cap \Lambda)$ to $\mathbb{C}/\sqrt{-1}\mathbb{Z}$, the graph map embedded $V_{t}/(\mathrm{Fil}^{a} \mathfrak{g}\cap \Lambda)$ into $\mathfrak{m}\times \mathbb{C}/\sqrt{-1}\mathbb{Z}$ as a closed subgroup. Then we get the following exact sequence:
$$0 \to \mathbb{C}/\sqrt{-1}\mathbb{Z} \to (\mathfrak{m} \times \mathbb{C}/\sqrt{-1}\mathbb{Z})/(V_{t}/(\mathrm{Fil}^{a} \mathfrak{g}\cap \Lambda)) \to \mathrm{A}_{\mathbb{C}} \to 0.$$
So we get a principal $\mathbb{C}/\sqrt{-1}\mathbb{Z}$-bundle on $\mathrm{A}_{\mathbb{C}}$. It determines a primitive element in $\oH^{1}(\mathrm{A}_{\mathbb{C}}, (\mathscr{O}^{an}_{\mathrm{A}_{\mathbb{C}}})^{*})$. By GAGA \cite{Se} and Appell-Humbert Theorem \cite[Chapter 1, Section 2]{Mu}, the group of primitive elements in $\oH^{1}(\mathrm{A}_{\mathbb{C}}, (\mathscr{O}^{an}_{\mathrm{A}_{\mathbb{C}}})^{*})$ is canonically isomorphic to $V_{a,\mathbb{C}}^{\check{}}/\Lambda_{a}^{\check{}}$. So we get a map $\phi_{t}$ from $\Lambda_{t}$ to $V_{a,\mathbb{C}}^{\check{}}/\Lambda_{a}^{\check{}}$. It is routine to check that the map is $\mathbb{Z}[\mathrm{Gal}(\mathbb{C}/\mathbb{R})]$-linear. Thus the $6$-tuple $(V_{a},\Lambda_{a}, V_{v}, \Lambda_{t} ,\phi_{v}, \phi_{t})$ is a marked polarizable lattice. It is routine to check that a morphism between two filtered polarizable lattices induces a morphism between two marked polarizable lattices.

The above construction defines an additive functor from the category $\mathcal{C}''$ of filtered polarizable lattices to the category $\mathbb{C}'$ of marked polarizable lattices. Combining the two constructions of functors defined in this section and the construction in the previous section, we get the following commutative diagram:
\begin{equation*}
   \xymatrix{
      \mathcal{C}  \ar[d]\ar[rd] &  \\
      \mathcal{C}''\ar[r]  & \mathcal{C}'
    }
\end{equation*}
Moreover, we have the following result.

\begin{prp}
All arrows in the above diagram are equivalences of categories.
\end{prp}

\begin{proof}
By Theorem \ref{ccc}, the arrow from $\mathcal{C}$ to $\mathcal{C}'$ is an equivalence of categories. So we only need to prove $\mathcal{C} \to \mathcal{C}''$ is essentially full and $\mathcal{C}'' \to \mathcal{C}'$ is faithful.

First we show $\mathcal{C}'' \to \mathcal{C}'$ is faithful.  Let $(\mathfrak{g}_{0,1},\Lambda_{1})$ and $(\mathfrak{g}_{0,2}, \Lambda_{2})$ be two filtered polarizable lattices. Let
    $$(V_{a,1},\Lambda_{a,1}, V_{v,1}, \Lambda_{t,1} ,\phi_{v,1}, \phi_{t,1})$$
and
    $$(V_{a,2},\Lambda_{a,2}, V_{v,2}, \Lambda_{t,2} ,\phi_{v,2}, \phi_{t,2})$$
be the marked polarizable lattices associated to them. Let  $\varphi$ be a morphism from $(\mathfrak{g}_{0,1},\Lambda_{1})$ to $(\mathfrak{g}_{0,2},\Lambda_{2})$. Suppose that $\varphi$ induces a zero map between the associated marked polarizable lattices, i.e., $\varphi$ is zero on $\mathrm{Fil}^{v}\mathfrak{g}_{0,1}$ and induces a zero map on $\mathrm{Fil}^{a}\mathfrak{g}_{1}\cap\Lambda_{1}$ and a zero map from $V_{a,1}\otimes_{\mathbb{R}}\mathbb{C}/\Lambda_{a,1}$ to $V_{a,2}\otimes_{\mathbb{R}}\mathbb{C}/\Lambda_{a,2}$. Then we have the following commutative diagram of short exact sequences of complex commutative algebraic groups:
    \begin{equation*}
        \xymatrix{
             0 \ar[r] & \mathrm{Fil}^{a}\mathfrak{g}_{1}/\mathrm{Fil}^{a}\mathfrak{g}_{1}\cap\Lambda_{1} \ar[d] \ar[r] & \mathfrak{g}_{1}/\Lambda_{1} \ar[d] \ar[r] & V_{a,1}\otimes_{\mathbb{R}}\mathbb{C}/\Lambda_{a,1} \ar[d] \ar[r] & 0  \\
          0 \ar[r] & \mathrm{Fil}^{a}\mathfrak{g}_{2}/\mathrm{Fil}^{a}\mathfrak{g}_{2}\cap\Lambda_{2}  \ar[r] & \mathfrak{g}_{2}/\Lambda_{2} \ar[r] & V_{a,2}\otimes_{\mathbb{R}}\mathbb{C}/\Lambda_{a,2} \ar[r] & 0,
        }
    \end{equation*}
the vertical maps are $\varphi\otimes_{\mathbb{R}}\mathbb{C}$. By our assumption, the first and the last vertical maps are zero. By snake lemma the middle vertical map induces a map from $V_{a,1}\otimes_{\mathbb{R}}\mathbb{C}/\Lambda_{a,1}$ to $\mathrm{Fil}^{a}V_{2}/\Lambda_{t,2}$. Since $V_{a,1}\otimes_{\mathbb{R}}\mathbb{C}/\Lambda_{a,1}$ is an abelian variety and $\mathrm{Fil}^{a}\mathfrak{g}_{2}/\mathrm{Fil}^{a}\mathfrak{g}_{2}\cap\Lambda_{2}$ is an affine group, the map must be zero. So $\varphi$ is zero.

Then we prove $\mathcal{C} \to \mathcal{C}''$ is essentially full. For any filtered polarizable lattice $(\mathfrak{g}_{0}, \Lambda)$, recall the notations in the construction of the additive functor $\mathcal{C}''\to \mathcal{C}'$ in this section. We have the following  exact sequence :
    $$0 \to \mathrm{Fil}^{v}\mathfrak{g} \to \mathfrak{n} \to \mathrm{A}_{\mathbb{C}} \to 0.$$
By Lefschetz theorem \cite[Chapter 1, Section 3]{Mu}, $\mathrm{A}_{\mathbb{C}}$ is an abelian variety. The exact sequence gives $\mathfrak{n}$ a structure of principal $\mathrm{Fil}^{v}\mathfrak{g}$-bundle on $\mathrm{A}_{\mathbb{C}}$. By GAGA \cite{Se}, there is a unique algebraic structure on $\mathfrak{n}$ compatible with the structure of principal $\mathrm{Fil}^{v}\mathfrak{g}$-bundle on $\mathrm{A}_{\mathbb{C}}$. One can define a unique algebraic group structure on $\mathfrak{n}$ with its structure of principal $\mathrm{Fil}^{v}\mathfrak{g}$-bundle of $\mathrm{A}_{\mathbb{C}}$ by \cite[Chapter 7, Theorem 5]{Ser}. Moreover, the argument in the proof of \cite[Chapter 7, Theorem 5]{Ser} can be used to show that there exists at least one complex commutative Lie group structure on $\mathfrak{n}$ compatible with its structure of principal $\mathrm{Fil}^{v}\mathfrak{g}$-bundle on $\mathrm{A}_{\mathbb{C}}$. So $\mathfrak{n}$ is algebraic. Similarly, the group $\mathfrak{m}$ is also algebraic. By Galois descent, the algebraic group $\mathfrak{n}\times_{\mathrm{A}_{\mathbb{C}}} \mathfrak{m}$ can be defined over $\mathbb{R}$. Hence the functor is essentially full.
\end{proof}

In particular, we have the following theorem:
\begin{thmp}\label{equivalence2}
The category of real connected commutative algebraic groups is naturally equivalent to the category of filtered polarizable lattices.
\end{thmp}

In fact, we can construct the inverse of the functor $\mathcal{C''} \to \mathcal{C'}$ as follows. Let $(V_{a},\Lambda_{a}, V_{v}, \Lambda_{t} ,\phi_{v}, \phi_{t})$ be a marked polarizable lattice. Let $\mathfrak{g}_{0}$ be the real vector space
    $$V_{a}\oplus V_{v}^{\check{}}\oplus (\Lambda_{t}^{\check{}}\otimes_{\mathbb{Z}}\mathbb{C})^{\sigma=1},$$
where $\sigma$ is the conjugate map in $\mathrm{Gal}(\mathbb{C}/\mathbb{R})$. Define a two-step filtration on $\mathfrak{g}_{0}$ by
    $$\mathrm{Fil}^{v}\mathfrak{g}_{0}=V_{v}^{\check{}}, \quad \quad \mathrm{Fil}^{a}V_{0}= V_{v}^{\check{}}\oplus (\Lambda_{t}^{\check{}}\otimes_{\mathbb{Z}}\mathbb{C})^{\sigma=1}.$$
Then the complex vector space $\mathfrak{g}$ is
    $$(V_{a}\otimes_{\mathbb{R}}\mathbb{C})\oplus (V_{v}^{\check{}}\otimes_{\mathbb{R}}\mathbb{C}) \oplus (\Lambda_{t}^{\check{}}\otimes_{\mathbb{Z}}\mathbb{C}).$$
The map $\phi_{v}$ naturally induces an anti-linear map $\phi_{v,\mathbb{C}}^{\check{}}$ from the complex space $V_{a}\otimes_{\mathbb{R}}\mathbb{C}$ to the complex space $\mathrm{Hom}_{\mathbb{C}}(V_{v}\otimes_{\mathbb{R}}\mathbb{C},\mathbb{C})$. We denote its restriction on $\Lambda_{a}$ also by $\phi_{v,\mathbb{C}}^{\check{}}$. The map $\phi_{t}$ naturally induces a $\mathbb{Z}$-module homomorphism $\phi_{t}^{\check{}}$ from the lattice $\Lambda_{a}$ to $(\Lambda_{t}^{\check{}}\otimes_{\mathbb{Z}}\mathbb{C})/(\sqrt{-1}\Lambda_{t}^{\check{}})$. Both $\phi_{v,\mathbb{C}}^{\check{}}$ and $\phi_{t}^{\check{}}$ are  $\mathrm{Gal}(\mathbb{C}/\mathbb{R})$-equivariant. Consider the graph $\Lambda_{\phi}$ of $\phi_{v,\mathbb{C}}^{\check{}}\oplus \phi_{t}^{\check{}}$ in
    $$(V_{a}\otimes_{\mathbb{R}}\mathbb{C})\oplus (V_{v}^{\check{}}\otimes_{\mathbb{R}}\mathbb{C}) \oplus ((\Lambda_{t}^{\check{}}\otimes_{\mathbb{Z}}\mathbb{C})/(\sqrt{-1}\Lambda_{t}^{\check{}})).$$
Denote its preimage in $\mathfrak{g}$ by $\Lambda$. Then the pair $(\mathfrak{g}_{0},\Lambda)$ is the filtered polarizable lattice associated with $(V_{a},\Lambda_{a}, V_{v}, \Lambda_{t} ,\phi_{v}, \phi_{t})$.

\section{Classification of simply connected commutative locally Nash Groups}\label{scclng}

In this section, we classify simply connected commutative locally Nash groups. First, for any real commutative algebraic group
$\mathsf{G}$, the identity component of $\mathsf{G}(\mathbb{R})$, denoted by $\mathsf{G}(\mathbb{R})^{0}$, has a natural  Nash group structure. Thus it naturally induces a locally Nash structure on the universal covering of $\mathsf{G}(\mathbb{R})^{0}$. Moreover, we have the following proposition.

\begin{prp}\label{surj}
Every simply connected commutative locally Nash group is locally Nash equivalent to the universal covering of $\mathsf{G}(\mathbb{R})^{0}$, for some real connected commutative algebraic group $\mathsf{G}$.
\end{prp}

\begin{proof}
Let $G$ be a simply connected commutative locally Nash group. As a Lie group, it is isomorphic to $\mathbb{R}^{n}$ for some $n$. Modulo a full lattice, we get a torus $T$ which is compact. The locally Nash structure on $G$ induces a Nash structure on $T$. In this way, $T$ becomes a Nash group.

By \cite[Theorem A]{HP1}, we can find a connected algebraic group $\mathsf{G}$ defined over $\mathbb{R}$ such that there exist open neighborhoods $U,V$ of the unity of $T$, open neighborhoods $U',V'$ of the unity of $\mathsf{G}(\mathbb{R})$ and a Nash homeomorphism $\varphi$ from $V$ to $V'$ satisfying:
\begin{itemize}
        \item $U+U = V$.
        \item $\varphi(U)=U', \varphi(V)=V'$.
        \item $\varphi(x+y)=\varphi(x)+\varphi(y)$ for any $x,y \in U$.
\end{itemize}
It induces an isomorphism between the Lie algebras of $T$ and $\mathsf{G}(\mathbb{R})$. So $\mathsf{G}$ is a real commutative group.

By shrinking $U,V$, we may identify them with open neighborhoods of the unity in $G$ and identify $U', V'$ with open neighborhoods of the unity in the universal covering of $\mathsf{G}(\mathbb{R})^{0}$. Moreover, we could assume $U$ (which is not necessarily semi-algebraic) is convex, which ensures $2^{n-1}V$ is contained in $2^{n}V$ for any natural number $n$. We extend $\varphi$ to a map $\varphi_{n}$ on $2^{n}V$ by
$$ \varphi_{n}(x)=2^{n}\varphi(\frac{x}{2^{n}}).$$ The map $\varphi_{n}$ is a homeomorphism from $2^{n}V$ to $2^{n}V'$. By induction, we could prove that $$\varphi_{n}|_{2^{n-1}V}=\varphi_{n-1}, \quad \varphi_{n}(x+y)=\varphi_{n}(x)+\varphi_{n}(y)$$for any $x, y\in 2^{n}U$. Putting all $n$ together, thus we get a commutative Lie group isomorphism from $G$ to the universal covering group of $\mathsf{G}(\mathbb{R})^{0}$. It is a locally Nash map around the unity with respect to the original locally Nash structure on $G$. So  it is a locally Nash equivalence between the simply connected group $G$ and the universal covering of $\mathsf{G}(\mathbb{R})^{0}$.
\end{proof}

Then we study locally Nash maps between two simply connected commutative locally Nash groups. Let $\mathsf{G}_{1}, \mathsf{G}_{2}$ be two real connected commutative algebraic groups and let $G_{1}, G_{2}$ be the universal coverings of
$\mathsf{G}_{1}(\mathbb{R})^{0}$ and $\mathsf{G}_{2}(\mathbb{R})^{0}$, respectively. Let $\varphi$ be a locally Nash map from $G_{1}$ to $G_{2}$.  Then we have the following proposition.


\begin{prp}\label{inj}
There exists a real connected commutative algebraic group $\mathsf{G}_{3}$ with an isogeny to $\mathsf{G}_{1}$, an algebraic map to $\mathsf{G}_{2}$ and a covering map from $G_{1}$ to $\mathsf{G}_{3}(\mathbb{R})^{0}$ commutating with $\varphi$. That is to say that we have the following commutative diagram
\begin{equation*}
  \xymatrix{
  G_{1} \ar[d] \ar[r]^{\varphi} & G_{2} \ar[d] \\
  \mathsf{G}_{3}(\mathbb{R})^{0} \ar[d] \ar[r] & \mathsf{G}_{2}(\mathbb{R})^{0}  \\
  \mathsf{G}_{1}(\mathbb{R})^{0} &    }
\end{equation*}
with the vertical arrow from $G_{1}$ to $\mathsf{G}_{3}(\mathbb{R})^{0}$ being a covering map, the vertical arrow from
$\mathsf{G}_{3}(\mathbb{R})^{0}$ to $\mathsf{G}_{1}(\mathbb{R})^{0}$ induced from an isogeny and the horizontal arrow from $\mathsf{G}_{3}(\mathbb{R})^{0}$ to $\mathsf{G}_{2}(\mathbb{R})^{0}$ induced from an algebraic group homomorphism.
\end{prp}

\begin{proof}
In this setting, we have a locally Nash map from the locally Nash group $G_{1}$ to $(\mathsf{G}_{1} \times_{\mathrm{Spec}(\mathbb{R})} \mathsf{G}_{2})(\mathbb{R})^{0}$. Its derivation is an injective Lie algebra homomorphism. The image is a semi-algebraic sets of the same dimension as $G_{1}$ by \cite[Proposition 2.8.8]{BCR}. Then let $\mathsf{G}_{3}$ be the Zariski closure of image of $G_{1}$. There exists a natural algebraic group structure on $\mathsf{G}_{3}$ and we have the following commutative diagram
\begin{equation*}
    \xymatrix{
      G_{1}  \ar[d]\ar[rd]^{\varphi}& \\
      \mathsf{G}_{3}(\mathbb{R})^{0}\ar@{^{(}->}[d] & G_{2} \ar[d]  \\
      (\mathsf{G}_{1}\times_{\mathrm{Spec}(\mathbb{R})}\mathsf{G}_{2})(\mathbb{R})^{0}\ar[r]\ar[d] & \mathsf{G}_{2}(\mathbb{R})^{0}.  \\
      \mathsf{G}_{1}(\mathbb{R})^{0}& \\
    }
\end{equation*}
Then we finish the proof.
\end{proof}

The following theorem follows immediately from Proposition \ref{inj}.

\begin{thmp}
The locally Nash equivalence classes of simply connected commutative locally Nash groups are one-to-one corresponding to the isogeny classes of real connected commutative algebraic groups.
\end{thmp}

In Sections \ref{rcag1} and \ref{rcag2}, we have classified the real connected commutative algebraic groups in terms of marked polarizable lattices and filtered polarizable lattices. To parametrize the isogeny classes, we only need to slightly modify the definition of filtered polarizable lattices. Consider the category defined as follows. A filtered polarizable rational space is a pair
$(\mathfrak{g}_{0},\Lambda_{r})$ where
\begin{itemize}
    \item $\mathfrak{g}_{0}$ is a real finite-dimensional vector space with a two-step filtration
    $$\mathrm{Fil}^{v} \mathfrak{g}_{0} \subseteq \mathrm{Fil}^{a} \mathfrak{g}_{0} \subseteq \mathfrak{g}_{0}.$$
    \item $\Lambda_{r}$ is a rational subspace in $\mathfrak{g}=\mathfrak{g}_{0}\otimes_{\mathbb{R}}\mathbb{C}$ stable under the natural $\mathrm{Gal}(\mathbb{C}/\mathbb{R})$-action on $\mathfrak{g}$
\end{itemize}
such that
\begin{itemize}
    \item $\mathrm{Fil}^{v} \mathfrak{g} \cap \Lambda_{r} = 0;$
    \item $(\mathrm{Fil}^{a} \mathfrak{g} \cap \Lambda_{r})\otimes_{\mathbb{Q}}\mathbb{C}= \mathrm{Fil}^{a} \mathfrak{g}/ \mathrm{Fil}^{v} \mathfrak{g};$
    \item $(\Lambda_{r}/(\mathrm{Fil}^{a} \mathfrak{g} \cap \Lambda_{r}))\otimes_{\mathbb{Q}} \mathbb{R} = \mathfrak{g} /\mathrm{Fil}^{a} \mathfrak{g} $
    \item there exists a positive definite symmetric form $S$ on $\mathfrak{g}_{0} /\mathrm{Fil}^{a} \mathfrak{g}_{0}$ satisfying $E(x_{1}+y_{1}\sqrt{-1},x_{2}+y_{2}\sqrt{-1}):= S(y_{1},x_{2})-S(x_{1},y_{2})$ takes rational values on $\Lambda_{r}/(\mathrm{Fil}^{a} \mathfrak{g} \cap \Lambda_{r})$.
\end{itemize}
A morphism from $(\mathfrak{g}_{0,1},\Lambda_{r,1})$ to $(\mathfrak{g}_{0,2},\Lambda_{r,2})$
is a real linear map  $\varphi$ from $\mathfrak{g}_{0,1}$ to $\mathfrak{g}_{0,2}$ such that
\begin{itemize}
    \item $\varphi_{\mathbb{C}}(\Lambda_{r,1}) \subseteq \Lambda_{r,2}$;
    \item $\varphi(\mathrm{Fil}^{v}\mathfrak{g}_{0,1}) \subseteq \mathrm{Fil}^{v}\mathfrak{g}_{0,2}$;
    \item $\varphi(\mathrm{Fil}^{a}\mathfrak{g}_{0,1}) \subseteq \mathrm{Fil}^{a}\mathfrak{g}_{0,2}$.
\end{itemize}
This is also an additive category.

For any simply connected commutative locally Nash group $G$, by Proposition \ref{surj} there exists a real connected commutative algebraic group $\mathsf{G}$ such that $G$ is locally Nash equivalent to the universal covering of $\mathsf{G}(\mathbb{R})^{0}$. Let $(\mathfrak{g}_{0}, \Lambda)$ be the filtered polarizable lattice associated with $\mathsf{G}$ and let
$\Lambda_{r}$ be $\Lambda\otimes_{\mathbb{Z}}\mathbb{Q}$. Then it is easy to check that $(\mathfrak{g}_{0}, \Lambda_{r})$ is a filtered polarizable rational space. By Proposition \ref{inj}, it is independent of the choice of $\mathsf{G}$. Furthermore, a locally Nash map between two simply connected commutative locally Nash groups induce a morphism in the category
$\mathcal{C}''$ by Proposition \ref{inj} and Theorem \ref{equivalence2}, while a morphism between two filtered polarizable lattices naturally induces a morphism between two associated filtered polarizable rational spaces. Thus this construction gives  a functor from the category of simply connected commutative locally Nash groups to the category of filtered polarizable rational spaces. We show that this functor actually establishes an equivalence between these two categories.

\begin{thmp}\label{simply}
The category of simply connected commutative locally Nash groups is equivalent to the category of filtered polarizable rational spaces.
\end{thmp}

\begin{proof}
We consider the functor constructed as above. For any filtered polarizable rational space $(\mathfrak{g}_{0},\Lambda_{r})$, choose a $\mathrm{Gal}(\mathbb{C}/\mathbb{R})$-stable full lattice $\Lambda$ in $\Lambda_{r}$. Then $(\mathfrak{g}_{0},\Lambda)$ is a filtered polarizable lattice. Let $\mathsf{G}$ be a real connected commutative algebraic group associated with
$(\mathfrak{g}_{0},\Lambda)$. Then $(\mathfrak{g}_{0},\Lambda_{r})$ is associated with the universal covering of  $\mathsf{G}(\mathbb{R})^{0}$, which is a simply connected commutative locally Nash group. So the functor is essentially full.

It is trivial that the functor is faithful. Let $G_{1}, G_{2}$ be two simply connected commutative locally Nash groups and
$(\mathfrak{g}_{0,1},\Lambda_{r,1})$ and $(\mathfrak{g}_{0,2},\Lambda_{r,2})$ be the correspoding filtered polarizable rational spaces, respectively. Choose $\mathrm{Gal}(\mathbb{C}/\mathbb{R})$-stable full lattices $\Lambda_{1},\Lambda_{2}$ in
$\Lambda_{r,1}$ and $\Lambda_{r,2}$, respectively. Let $\mathsf{G}_{1}, \mathsf{G}_{2}$ be real connected commutative algebraic groups corresponding to $(\mathfrak{g}_{0,1},\Lambda_{1})$ and $(\mathfrak{g}_{0,2},\Lambda_{2})$, respectively. Then we have the covering maps from $G_{1}, G_{2}$ to $\mathsf{G}_{1}(\mathbb{R})^{0},\mathsf{G}_{2}(\mathbb{R})^{0}$, respectively. For any morphism $\varphi$ from $(\mathfrak{g}_{0,1},\Lambda_{r,1})$ to $(\mathfrak{g}_{0,2},\Lambda_{r,2})$, choose an integer $N$ large enough so that $N\varphi(\Lambda_{1})\subseteq \Lambda_{2}$. Then $N\varphi$ corresponds to a homomorphism $\varphi'$ from $\mathsf{G}_{1}$ to $\mathsf{G}_{2}$. The map $\varphi'$ induces a locally Nash map from $G_{1}$ to $G_{2}$ denoted also by $\varphi'$. The locally Nash map $\frac{\varphi'}{N}$ then corresponds to $\varphi$. So the functor is fully faithful and we finish the proof of this theorem.
\end{proof}

For the convenience of computation, we introduce the notions of split dimensions. For any marked polarizable lattice $(V_{a},\Lambda_{a}, V_{v}, \Lambda_{t}, \phi_{v}, \phi_{t})$ and the associated algebraic group $\mathsf{G}$, we define the real vector space $U_{v}$ to be the image of $\phi_{v}$ and define $U_{t}$ to be the rational vector space spanned by the preimage of $\mathrm{im} \phi_{t}$ in $(V_{a}\otimes_{\mathbb{R}}\mathbb{C})^{\check{}}$. By the definition of marked polarizable lattices, the space $U_{t}$ is stable under $\mathrm{Gal}(\mathbb{C}/\mathbb{R})$-action and we denote $U_{t,+}$ (resp. $U_{t,-}$) the eigenspace of $\sigma$ with eigenvalue $1$ (resp. $-1$), where $\sigma$ is the conjugate map in $\mathrm{Gal}(\mathbb{C}/\mathbb{R})$. We define the split dimensions by the following formulas:
$$d_{v} = \dim_{\mathbb{R}} (\ker \phi_{v}),\quad d_{t,+} = \mathrm{rank}(\ker \phi_{t})^{\sigma=1}, \quad d_{t,-}=\mathrm{rank}(\ker \phi_{t})^{\sigma=-1}.$$
We call $d_{v}$ (resp. $d_{t,+}$, $ d_{t,-}$) the split additive dimension (resp. the split multiplicative dimension, the split twisted multiplicative dimension). By the above classification, Corollary \ref{antparametrization} and \cite[Theorem 3.4]{Br}, the number $d_{v}$ (resp. $d_{t,+}$, $ d_{t,-}$) denotes the dimension of a maximal direct summand (which is not necessarily unique but of the same dimension), which is a direct sum of additive groups (resp. multiplicative groups, twisted multiplicative groups), in the universal locally Nash covering group of $\mathsf{G}(\mathbb{R})^{0}$. Thus we could define these three split dimensions for any connected commutative locally Nash group, which is determined by its universal covering.

\section{Classification of connected commutative locally Nash groups}\label{cclng}

In this section, we will give a complete classification of both connected commutative locally Nash groups and connected commutative Nash groups. We introduce a new category $\mathcal{D}$ as follows. An object in $\mathcal{D}$ is a triple
$(\mathfrak{g}_{0}, \Lambda_{r}, \Gamma)$, where $(\mathfrak{g}_{0},\Lambda_{r})$ is a filtered polarizable rational space as defined in Section \ref{scclng} and $\Gamma$ is a discrete lattice in $\mathfrak{g}_{0}$.  A morphism from $(\mathfrak{g}_{0,1},\Lambda_{r,1},\Gamma_{1})$ to $(\mathfrak{g}_{0,1},\Lambda_{r,2},\Gamma_{2})$ in the category $\mathcal{D}$ is a morphism $\varphi$ from $(\mathfrak{g}_{0,1},\Lambda_{r,1})$ to $(\mathfrak{g}_{0,2},\Lambda_{r,2})$ in the category of filtered polarizable rational spaces such that $\varphi(\Gamma_{1})\subseteq \Gamma_{2}$.

For any connected commutative locally Nash group $G$, consider its universal covering $\tilde{G}$. The latter is a simply connected commutative locally Nash group. Let $(\mathfrak{g}_{0},\Lambda_{r})$ be the filtered polarizable rational space associated with $\tilde{G}$. Identify $\tilde{G}$ with $\mathfrak{g}_{0}$ as a Lie group. Then the kernel of the covering map from $\tilde{G}$ to $G$ is a discrete lattice in $\mathfrak{g}_{0}$, denoted by $\Gamma$. In this way, we get a functor from the category of connected commutative locally Nash groups to the category $\mathcal{D}$.

\begin{thm}
The category of connected commutative locally Nash groups is equivalent to the category $\mathcal{D}$ as defined above.
\end{thm}

\begin{proof}
A continuous homomorphism between two connected commutative locally Nash groups is a locally Nash map if and only if it induces a locally Nash map between their universal coverings. The theorem then reduces to Theorem \ref{simply}.
\end{proof}

Now we study the conditions under which a connected commutative locally Nash group associated with the triple $(\mathfrak{g}_{0},\Lambda_{r},\Gamma)$ becomes a Nash group. To state the results, we introduce a subspace $\mathfrak{g}_{0}^{\mathrm{enca}}$ in $\mathfrak{g}_{0}$ called the essentially non-compact affine subspace as
 $$\mathfrak{g}_{0}^{\mathrm{enca}}:=\mathrm{Fil}^{v}\mathfrak{g}_{0}\oplus(\sqrt{-1}(\Lambda_{r}^{\sigma=-1}\otimes_{\mathbb{Q}}\mathbb{R})\cap\mathrm{Fil}^{a}\mathfrak{g}_{0}).$$
 Recall that $\sigma$ is the conjugate map in $\mathrm{Gal}(\mathbb{C}/\mathbb{R})$. First we give a sufficient condition.

\begin{prpt}\label{sufficientforNash}
Let $G$ be the connected commutative locally Nash group associated with the triple $(\mathfrak{g}_{0},\Lambda_{r},\Gamma)$. If the lattice $\Gamma$ spans the vector space $\mathfrak{g}_{0}/\mathfrak{g}_{0}^{\mathrm{enca}}$, then $G$ is a Nash group.
\end{prpt}

\begin{proof}
Since a Nash group modulo a discrete normal subgroup is also Nash, without loss of generality we may assume that the lattice $\Gamma$ is canonically isomorphic to a full discrete lattice $\Gamma'$ in $\mathfrak{g}_{0}'=\mathfrak{g}_{0}/\mathfrak{g}_{0}^{\mathrm{enca}}$.
Let $\lambda_{1}, \cdots \lambda_{m}$ be a set of generators of $\Gamma'$ and write
    $$W=\left\{\sum_{i=1}^{m}a_{i}\lambda_{i}\Big|-\frac{1}{2}< a_{i}<\frac{1}{2}, 1\leq i \leq m\right\}.$$

Let $\mathsf{G}$ be a real connected commutative algebraic group such that $\mathsf{G}(\mathbb{R})^{0}$ and $G$ have the same universal covering. We can assume that
    $$\mathsf{G}_{\mathrm{aff}}=\mathsf{V}\times \mathsf{T}_{1}\times \mathsf{T}_{2},$$
where $\mathsf{T}_{1}=\mathsf{G}_{m,\mathbb{R}}^{n_{1}}$ and $\mathsf{T}_{2}=\mathrm{SO}(2,\mathbb{R})^{n_{2}}$. Then
$\mathsf{G}$ fits the following exact sequence
    $$0 \to \mathsf{V}\times \mathsf{T}_{1} \to \mathsf{G} \to \mathsf{T}_{2}\times \mathsf{A}\to 0.$$
As a $\mathsf{V}\times \mathsf{T}_{1}$ torsor over $\mathsf{T}_{2}\times \mathsf{A}$, there exists an affine neighborhood $\mathsf{U}$ around $0$ in $\mathsf{T}_{2}\times \mathsf{A}$ such that the preimage of $\mathsf{U}$ in $\mathsf{G}$ is the product (not as groups but as real algebraic varieties) $\mathsf{V} \times \mathsf{T}_{1}\times \mathsf{U}$.

Passing to real points, we identify $\mathfrak{g}_{0}^{\mathrm{enca}}$ with the real points of $\mathsf{V}\times \mathsf{T}_{1}$, and identify $\mathfrak{g}_{0}'$ with the universal locally Nash covering of the identity component of $\mathsf{T}_{2}\times \mathsf{A}$ and the identity component of the preimage of $\mathsf{U}(\mathbb{R})^{0}$ in $\mathfrak{g}_{0}'$ is an affine Nash open neighborhood $U'$ around $0$. We can shrink $U'$ further to make it contained in $W$. Denote $\tilde{U}$ the preimage of $U'$ in $\tilde{G}$. By our assumption $\tilde{U}$ is canonically isomorphic to an affine Nash submanifold of $G$ and it is a $\mathfrak{g}_{0}^{\mathrm{enca}}$-torsor over $U'$. Hence the locally Nash group $G$ is a $\mathfrak{g}_{0}^{\mathrm{enca}}$-torsor over $\mathfrak{g}_{0}'/\Gamma'$ and $\mathfrak{g}_{0}'/\Gamma'$ is compact. Then by translation we get a finite affine Nash covering from $\tilde{U}$.
\end{proof}

We still need to show the condition is also necessary. Write
$\Gamma_{r}=\Gamma\cap \Lambda_{r}$ and $G_{r}=\tilde{G}/\Gamma_{r}$. Then $G_{r}$ is a locally Nash covering of $G$. For a given affine open Nash submanifold $U$ of $G$, we define $U_{r}$ to be the preimage of $U$ in $G_{r}$. We have the following lemma.

\begin{lemt}\label{lemmaforNash1}
Let $U_{r}^{0}$ be a connected component of $U_{r}$. Then the open locally Nash submanifold $U_{r}^{0}$ is canonically isomorphic to $U$.
\end{lemt}

\begin{proof}
Let $\Gamma_{U_{r}}$ be the group
$$\{\gamma \in \Gamma/\Gamma_{r}| \gamma+U_{r}^{0}=U_{r}^{0}\}$$
We only need to show it is trivial.

Let $\mathsf{G}$ be a real connected commutative algebraic groups such that $G_{r}$ is a covering group of $\mathsf{G}(\mathbb{R})^{0}$. Let $\mathsf{G}_{U_{r}}$ be the Zariski closure of $\Gamma_{U_{r}}$ in $\mathsf{G}$. It is a real commutative algebraic group. Since $\mathsf{G}$ acts algebraically on itself, any rational function on $\mathsf{G}$ invariant under $\Gamma_{U_{r}}$-action is also invariant under $\mathsf{G}_{U_{r}}$-action.

Let $f_{1},f_{2}, \cdots, f_{m}$ be a family of Nash coordinate functions on $U_{r}^{0}$. Since $U_{r}^{0}$ is connected, these functions are algebraic over the rational function field $K(\mathsf{G})$ for the algebraic group $\mathsf{G}$. Let $g_{1}, g_{2}, \cdots, g_{m'}$ be the coefficient functions of minimal polynomials of $f_{1}, \cdots ,f_{m}$ in $K(\mathsf{G})$. By our assumption $g_{1}, g_{2}, \cdots, g_{m'}$ are stable under $\mathsf{G}_{U_{r}}(\mathbb{R})$. So for any $x \in U_{r}^{0}$, there exists an open (in the sense of ordinary topology) subset $U_{0}$ of $\mathsf{G}_{U_{r}}(\mathbb{R})$ such that for any $y \in x+U_{0}$ the coordinate $(f_{1}(y),f_{2}(y), \cdots, f_{m}(y))$ satisfy the same algebraic equations. So $U_{0}$ must be of dimension $0$. Thus $\mathsf{G}_{U_{r}}$ is of dimension $0$ and hence $\Gamma_{U_{r}}$ is finite. On the other hand
$\Gamma_{U_{r}}$ is a subgroup of a free abelian group. It follows that $\Gamma_{U_{r}}$ is trivial.
\end{proof}

Thanks to this lemma, we identify $U$ and such $U_{r}^{0}$ in the following discussion. As in the proof of Proposition \ref{sufficientforNash}, we view $\mathfrak{g}_{0}^{\mathrm{enca}}$ as a closed locally Nash subgroup in $\tilde{G}$, hence in $G_{r}$. Thus we have the following lemma.

\begin{lemt}\label{lemmaforNash2}
The closure of the image of $U$ in $G_{r}/\mathfrak{g}_{0}^{\mathrm{enca}}$ is compact.
\end{lemt}

\begin{proof}
Thanks to Lemma \ref{lemmaforNash1}, we view $U$ as an affine open submanifold of $G_{r}$ and we denote by $U''$ its image in $G_{r}/\mathfrak{g}_{0}^{\mathrm{enca}}$. Let $\mathsf{G}$ be a real connected commutative algebraic group such that $G_{r}$ is a covering group of $\mathsf{G}(\mathbb{R})^{0}$. Thus $\mathsf{G}(\mathbb{R})^{0}/\mathfrak{g}_{0}^{\mathrm{enca}}$ is a compact affine Nash group and we have a natural locally Nash map from $G_{r}/\mathfrak{g}_{0}^{\mathrm{enca}}$ to $\mathsf{G}(\mathbb{R})^{0}/\mathfrak{g}_{0}^{\mathrm{enca}}$.

For any point $x$ in $\mathsf{G}(\mathbb{R})^{0}/\mathfrak{g}_{0}^{\mathrm{enca}}$, we can find an open neighbourhood $U_{x}$ such that the closure of any connected component of its preimage in $G_{r}/\mathfrak{g}_{0}^{\mathrm{enca}}$ is compact. By the compactness of $\mathsf{G}(\mathbb{R})^{0}/\mathfrak{g}_{0}^{\mathrm{enca}}$, we may choose a finite set $J$ such that $\mathsf{G}(\mathbb{R})^{0}/\mathfrak{g}_{0}^{\mathrm{enca}}$ is covered by the family of open subsets $\{U_{x}\}_{x\in J}$.

View $U$ as an affine open submanifold of $G_{r}$, then the naturally induced map from $U$ to $\mathsf{G}(\mathbb{R})^{0}/\mathfrak{g}_{0}^{\mathrm{enca}}$ is Nash. The preimage $C_{x}$ of a point $x$ in $U$ is affine. So $C_{x}$ has finitely many connected components of the form $$\{(y+\mathfrak{g}_{0}^{\mathrm{enca}})\cap C_{x}|y\in I_{x}\},$$ where $I_{x}$ is a finite subset in $U$. Denote $\overline{y}$ the image of $y$ in $G_{r}/\mathfrak{g}_{0}^{\mathrm{enca}}$ and denote $U_{y}''$ the connected component of the preimage of $U_{x}$ in $\mathsf{G}(\mathbb{R})^{0}/\mathfrak{g}_{0}^{\mathrm{enca}}$ containing $\overline{y}$. Then the family of open sets $$\{U_{y}''|x\in J, y\in I_{x}\}$$cover $U''$. So the closure of $U''$ is compact.
\end{proof}

\begin{thm}
Let $G$ be the connected commutative locally Nash group associated with the triple $(\mathfrak{g}_{0},\Lambda_{r},\Gamma)$. Then the group $G$ is Nash if and only if $\Gamma$ spans the vector space $\mathfrak{g}_{0}/\mathfrak{g}_{0}^{\mathrm{enca}}$.
\end{thm}

\begin{proof}
We only need to prove the only if part. We prove it by contradiction. Let $\mathfrak{g}_{0}^{\Gamma}$ denote the real vector subspace in $\mathfrak{g}_{0}$ spanned by $\Gamma$. If $\mathfrak{g}_{0}/\mathfrak{g}_{0}^{\mathrm{enca}}$
is not spanned by $\Gamma$, then the vector space $\mathfrak{g}_{0}/(\mathfrak{g}_{0}^{\mathrm{enca}}+\mathfrak{g}_{0}^{\Gamma})$ is not zero. Thus we have a natural map from $G_{r}$ to $\mathfrak{g}_{0}/(\mathfrak{g}_{0}^{\mathrm{enca}}+\mathfrak{g}_{0}^{\Gamma})$.

On the other hand, let $U_{1}, U_{2}, \cdots ,U_{m}$ be a finite affine Nash covering of $G$. By Lemma \ref{lemmaforNash1}, $G_{r}$ is covered by the following family of affine Nash open sets
    $$\{U_{i}+ \gamma|1\leq i \leq m , \gamma \in \Gamma/\Gamma_{r}\}.$$
Notice that the affine open submanifold $U_{i}$ in $G_{r}$ is viewed as a connected component of the preimage of $U_{i}$ thanks to Lemma \ref{lemmaforNash1}. It is obvious that the image of $U_{i}+\gamma$ in $\mathfrak{g}_{0}/(\mathfrak{g}_{0}^{\mathrm{enca}}+\mathfrak{g}_{0}^{\Gamma})$ is independent of $\gamma$. Since the natural map from $G_{r}$ to $\mathfrak{g}_{0}/(\mathfrak{g}_{0}^{\mathrm{enca}}+\mathfrak{g}_{0}^{\Gamma})$ factors through the group $G_{r}/\mathfrak{g}_{0}^{\mathrm{enca}}$, the image of $U_{i}$ in $\mathfrak{g}_{0}/(\mathfrak{g}_{0}^{\mathrm{enca}}+\mathfrak{g}_{0}^{\Gamma})$ is contained in a compact set. Thus the image of $G_{r}$ in $\mathfrak{g}_{0}/(\mathfrak{g}_{0}^{\mathrm{enca}}+\mathfrak{g}_{0}^{\Gamma})$ is also contained in a compact set. However, the natural map from $G_{r}$ to $\mathfrak{g}_{0}/(\mathfrak{g}_{0}^{\mathrm{enca}}+\mathfrak{g}_{0}^{\Gamma})$ is surjective, this leads to a contradiction with the assumption that $\mathfrak{g}_{0}/(\mathfrak{g}_{0}^{\mathrm{enca}}+\mathfrak{g}_{0}^{\Gamma})$ is not zero.
\end{proof}

\section{Affineness and toroidal affineness}\label{affinetoroidal}

We determine the conditions under which a connected commutative locally Nash group becomes affine (resp. toroidal affine) using triples in the category $\mathcal{D}$. Let $G$ be a connected commutative locally Nash group and let $(\mathfrak{g}_{0},\Lambda_{r},\Gamma)$ be the triple associated with $G$.

\begin{prp}\label{affine}
The group $G$ is affine if and only if $\Lambda_{r}\cap\mathfrak{g}_{0}=\Gamma \otimes_{Z}\mathbb{Q}$.
\end{prp}

\begin{proof}
We prove the if part first. It is known that an affine Nash group is a finite cover of a real algebraic group. We first suppose that $G$ is the identity component of a real connected commutative algebraic group $\mathsf{G}$. Thus the filtered polarizable lattice associated with $\mathsf{G}$ is $(\mathfrak{g}_{0},\Lambda)$ with $\Lambda_{r}=\Lambda \otimes_{Z} \mathbb{Q}$. As a Lie group, the group $G$ is isomorphic to $\mathfrak{g}_{0}/(\Lambda \cap \mathfrak{g}_{0})$. So we can identify $\Gamma$ with  $\Lambda \cap \mathfrak{g}_{0}$. In this case we get $\Lambda_{r} \cap \mathfrak{g}_{0} = \Gamma \otimes_{Z}\mathbb{Q}$. The property is stable under isogenies, so it holds for general affine Nash groups.

Conversely, if $(\mathfrak{g}_{0},\Lambda_{r},\Gamma)$ satisfies $V_{r} \cap \mathfrak{g}_{0} = \Gamma \otimes_{Z}\mathbb{Q}$, we take a finitely generated free $\mathbb{Z}$-submodule $\Lambda_{-}$ in the rational space $V_{r}^{\sigma=-1}$ of maximal possible rank, where as before, $\sigma$ is the conjugate map in $\mathrm{Gal}(\mathbb{C}/\mathbb{R})$. Let $\mathsf{G}$ be the  algebraic group associated with $(\mathfrak{g}_{0},\Lambda_{-}\oplus\Gamma)$. Then $G$ is the identity component of $\mathsf{G}$ and is naturally affine.
\end{proof}

In one-dimensional case, our criterion for affineness coincides with the classification of "embeddable" Nash groups in \cite{MS}.

Recall that an affine Nash group is said to be anti-linear if it has no quotient Nash group which is almost linear and positive dimensional. A toroidal affine Nash group \cite{Can} is defined to be a connected anti-linear affine Nash group.
The split dimensions $d_v, d_{t,+}, d_{t,-}$ of a connected commutative locally Nash group have been defined in the end of Section \ref{scclng}.

\begin{prp}
The group $G$ is toroidal affine if and only if $\Lambda_{r}\cap \mathfrak{g}_{0}= \Gamma \otimes_{Z}\mathbb{Q}$ and $d_{t,+}=d_{t,-} =d_{v}=0$.
\end{prp}

\begin{proof}
As in the proof of Proposition \ref{affine}, we may assume that the affine Nash group $G$ is Nash equivalent to the identity component of a real connected commutative algebraic group $\mathsf{G}$. The filtered polarizable lattice $(\mathfrak{g}_{0}, \Lambda)$ associated with $\mathsf{G}$ satisfies that $\Lambda_{r}=\Lambda\otimes_{\mathbb{Z}}\mathbb{Q}$. By \cite[Lemmas 3.3 and 3.5]{Can}, $G$ is a toroidal affine Nash group if and only if $\mathsf{G}$ is an anti-affine algebraic group. By \cite[Proposition 2.11]{Br}, $\mathsf{G}$ is anti-affine if and only if $d_{t,+}=d_{t,-}=d_{v}=0$.
\end{proof}

\section{Examples in dimension $1$ and $2$}\label{lowdimension}

In this section, we show that our classification matches the previous work in \cite{MS} and \cite{BVO} for one-dimensional and two-dimensional cases. Recall the notations of split dimensions at the end of  Section \ref{scclng}. It is obvious that the sum of these three split dimensions is no larger than the dimension of the locally Nash group. In one-dimensional case, there are
$4$-possible cases for split dimensions (cf. \cite[Theorem 1]{MS}):
\begin{itemize}
    \item $d_{v}=1$. The corresponding simply connected locally Nash group is $\mathbb{R}$ with the ordinary additive structure.
    \item $d_{t,+}=1$. The corresponding simply connected locally Nash group is $\mathbb{R}_{+}$ with the ordinary multiplicative structure.
    \item $d_{t,-}=1$. The corresponding simply connected locally Nash group is the universal covering of $\mathrm{SO}(2,\mathbb{R})$.
    \item $d_{v}=d_{t,+}=d_{t,-}=0$. The corresponding simply connected locally Nash groups are the universal coverings of the identity component of real elliptic curves.
\end{itemize}
For each class, two discrete lattices $\Gamma_{1}, \Gamma_{2}$ define the same locally Nash group if and only if $\Gamma_{2}=a\Gamma_{1}$ with $a \in \mathbb{Q}^{\times}$.

In two-dimensional case, there are $3$-possible cases for split dimensions (cf. \cite[Proposition 7.4]{BVO}):
\begin{itemize}
    \item $d_{v}+d_{t,+}+d_{t,-}=2$. The corresponding simply connected locally Nash group is the direct sum of $\mathbb{R}$, $\mathbb{R}_{+}$ or the universal covering of $\mathrm{SO}(2,\mathbb{R})$.
    \item $d_{v}+d_{t,-}+d_{t,+}=1$. The corresponding simply connected locally Nash group is the direct sum of $\mathbb{R}$, $\mathbb{R}_{+}$ or the universal covering of $\mathrm{SO}(2,\mathbb{R})$ with the universal covering of the identity component of a real elliptic curve.
    \item $d_{t,+}=d_{t,-}=d_{v}=0$ and $U_{t,-}=U_{t,+}=U_{v}=0$.  The corresponding simply connected locally Nash group is the universal covering of the identity component of a product of two real elliptic curves or of a simple two-dimensional real abelian variety.
    \item $d_{t,+}=d_{t,-}=d_{v}=0$ and $U_{t,-},U_{t,+}, U_{v}$ are not all trivial. The corresponding simply connected locally Nash group is a non-split extension of $\mathbb{R}$, $\mathbb{R}_{+}$ or the universal covering of $\mathrm{SO}(2,\mathbb{R})$ by the universal covering of the identity component of a real elliptic curve.
\end{itemize}
In the last case, we suppose the real elliptic curve is of real period $1$ and of imaginary period $a\sqrt{-1}$($a \in \mathbb{R}_{+}$). If $U_{v}\neq 0$, the non-split extension is unique because $\dim U_{v}=1$. If $U_{t,+}\neq 0$ (resp. $U_{t,-}\neq 0$), the non-split extensions can be parametrized by an element in
\[
(((\mathfrak{g}/\mathrm{Fil}^{a}\mathfrak{g})^{\check{}}/\Lambda_{r}^{\check{}})^{\sigma =1}\backslash\{0\})/\mathbb{Q}^{*} \,\,\text{(resp. $(((\mathfrak{g}/\mathrm{Fil}^{a}\mathfrak{g})^{\check{}}/\Lambda_{r}^{\check{}})^{\sigma =-1}\backslash\{0\})/\mathbb{Q}^{*}) $}
\]
in multiplicative (resp. twisted multiplicative) case.

We show that the classification in \cite{BVO} matches ours. The only nontrivial case is the last one. We review the constructions in \cite{BVO}. First recall some facts about Weierstrass $\wp$($\zeta$, $\sigma$)-functions. Let
$$\wp_{a\sqrt{-1}}(z)=\frac{1}{z^{2}}+\sum_{(m,n) \neq 0}\left(\frac{1}{(z-m - na\sqrt{-1})^{2}}-\frac{1}{(m +na\sqrt{-1})^{2}}\right),$$
$$\zeta_{a\sqrt{-1}}(z)=\frac{1}{z}+\sum_{(m,n) \neq 0}\left(\frac{1}{z-m - na\sqrt{-1}}+\frac{1}{m +na\sqrt{-1}}+\frac{z}{(m +na\sqrt{-1})^{2}}\right),$$
$$\sigma_{a\sqrt{-1}}(z)=z\prod_{(m,n) \neq 0}\left(1-\frac{z}{m +na\sqrt{-1}}\right)e^{z/(m+na\sqrt{-1})+(z/(m+na\sqrt{-1}))^{2}/2}.$$
We have the following (quasi)-periodic relations for them:
$$\wp_{a\sqrt{-1}}(z+m +na\sqrt{-1})=\wp_{a\sqrt{-1}}(z),$$
$$\zeta_{a\sqrt{-1}}(z+m +na\sqrt{-1})=\zeta_{a\sqrt{-1}}(z)+m\eta_{a\sqrt{-1}}(1)+n\eta_{a\sqrt{-1}}(a\sqrt{-1}),$$
\begin{equation*}
\begin{split}
\sigma_{a\sqrt{-1}}(z+m +na\sqrt{-1})&=(-1)^{(m,n)}e^{(m\eta_{a\sqrt{-1}}(1)+n\eta_{a\sqrt{-1}}(a\sqrt{-1}))z}\cdot\\
&e^{(m\eta_{a\sqrt{-1}}(1)+n\eta_{a\sqrt{-1}}(a\sqrt{-1}))(m+n\sqrt{-1})/2}\sigma_{a\sqrt{-1}}(z),
\end{split}
\end{equation*}
where $\eta_{a\sqrt{-1}}(1)$ is a real number, $\eta_{a\sqrt{-1}}(a\sqrt{-1})$ is a pure imaginary number, and they satisfy the following Legendre relations:
$$a\eta_{a\sqrt{-1}}(1)\sqrt{-1}-\eta_{a\sqrt{-1}}(a\sqrt{-1})=2\pi\sqrt{-1}.$$
For any $\xi \in \mathbb{C}$, define meromorphic functions $\tilde{\sigma}_{a\sqrt{-1},\xi}$ as (Note that the sign is a little different from \cite{BVO})
$$\tilde{\sigma}_{a\sqrt{-1},\xi} = \frac{\sigma_{a\sqrt{-1}}(z+\xi)}{\sigma_{a\sqrt{-1}}(z)}.$$
We have the following quasi-periodic relations:
$$\tilde{\sigma}_{a\sqrt{-1},\xi}(z+m+na\sqrt{-1})= e^{(m\eta_{a\sqrt{-1}}(1)+n\eta_{a\sqrt{-1}}(a\sqrt{-1}))\xi}\tilde{\sigma}_{a\sqrt{-1},\xi}(z).$$

The constructions of non-split extensions in \cite{BVO} can be summarized as follows. In the paper, the authors constructed non-split extensions of $\mathsf{H}(\mathbb{C})$ by $\mathsf{A}(\mathbb{C})$, where $\mathsf{H}$ is a real additive (resp. multiplicative and twisted multiplicative) group and $\mathsf{A}$ is a real elliptic curve. By Galois descent theory, the non-split extensions can be descended into non-split extensions of $\mathbb{R}$ (resp. $\mathbb{R}_{+}$ and $\mathrm{SO}(2,\mathbb{R})$) by $\mathsf{A}(\mathbb{R})^{0}$. Thus the non-split extensions could be pulled back to their universal coverings. Let $\Pi_{0}=\mathbb{Z}\times \{0\}$ in $\mathbb{R}^{2}$ and $\Pi=(\mathbb{Z}+a\sqrt{-1}\mathbb{Z})\times \{0\}$ in $\mathbb{C}^{2}$. We write
\begin{itemize}
  \item Additive case: $\Pi_{0}'=\{0\}\times \{0\}$ in $\mathbb{R}^{2}$ and $\Pi'=\{0\}\times \{0\}$ in $\mathbb{C}^{2}$;
  \item Multiplicative case: $\Pi'_{0}=\{0\}\times \{0\}$ and $\Pi'=\{0\}\times 2\pi\sqrt{-1}\mathbb{Z}$;
  \item Twisted multiplicative case: $\Pi'_{0}= \{0\}\times 2\pi\mathbb{Z}$ and $\Pi'=\{0\} \times 2\pi\mathbb{Z}$.
\end{itemize}
The quotient group $\mathbb{C}^{2}/\Pi'\times \Pi$ fits in the following exact sequence:
$$0 \to \mathbb{C}/\Pi' \to \mathbb{C}^{2}/\Pi'\times \Pi \to \mathbb{C}/\Pi \to 0.$$
Identify $\mathbb{C}/\Pi'$ with $\mathsf{H}(\mathbb{C})$ (by exponential map in multiplicative and twisted multiplicative cases) and $\mathbb{C}/\Pi$ with $\mathsf{A}(\mathbb{C})$ by the Weierstrass $\wp$-function with period lattice $\Pi$.

We construct the dual of $\mathbb{C}/\Pi$ as follows. First, we identify the space of anti-linear maps on $\mathbb{C}$ with the space $\mathbb{C}$ by sending a complex number $\alpha$ to an anti-linear map $\frac{\varphi_{\alpha}}{a}$ on $\mathbb{C}$ defined by $\varphi_{\alpha}(z)=\alpha\bar{z}$. Then we use the map from $\mathbb{C}/\Pi$ to $\Hom(\Pi, \mathrm{U}(1))$ which sends $\alpha$ to the map $\lambda \mapsto e^{2\frac{\pi}{a} \Im(\alpha \bar{\lambda})\sqrt{-1}}$ to identify $\mathbb{C}/\Pi$ with the dual of $\mathbb{C}/\Pi$. By Appell-Humbert Theorem \cite[Chapter 1, Section 2]{Mu}, we have
$$(((\mathfrak{g}/\mathrm{Fil}^{a}\mathfrak{g})^{\check{}}/\Lambda_{r}^{\check{}})^{\sigma =1}\backslash\{0\})/\mathbb{Q}^{*}= ((\mathbb{R}/\mathbb{Q})\backslash\{0\})/\mathbb{Q}^{*},$$
$$(((\mathfrak{g}/\mathrm{Fil}^{a}\mathfrak{g})^{\check{}}/\Lambda_{r}^{\check{}})^{\sigma =-1}\backslash\{0\})/\mathbb{Q}^{*}= ((\sqrt{-1}\mathbb{R}/a\sqrt{-1}\mathbb{Q})\backslash\{0\})/\mathbb{Q}^{*}.$$
On an open dense subset of $\mathbb{C}^{2}/\Pi'\times \Pi$, there is a bi-holomorphic map (which is not a group homomorphism) from $\mathbb{C}^{2}/\Pi'\times \Pi$ to  $\mathbb{C}/\Pi\times \mathbb{C}/\Pi'$ by
$$(u,v) \to ([\wp(u):\wp'(u):1], g(v)\ominus f(u))$$
where the function $g(v)=v$ (resp. $g(v)=e^{v}$ and $g(v)=e^{\sqrt{-1}v}$), $f(u)=\zeta_{a\sqrt{-1}}(u)$ (resp. $f(u)=\sigma_{a\sqrt{-1},\xi}(u)$ and $f(u)=\sigma_{a\sqrt{-1},\xi\sqrt{-1}}(u)$) and $\oplus$ is the addition map on $\mathsf{H}(\mathbb{C})$ in the additive case (resp. in the multiplicative case and in the twisted multiplicative case). The map induces a real algebraic structure on this open set. By translation, we get a real algebraic structure on whole $\mathbb{C}^{2}/\Pi'\times \Pi$. In this way we view $\mathbb{C}^{2}/\Pi'\times \Pi$ as the set of complex points of a real algebraic group $\mathsf{G}$. The real algebraic group $\mathsf{G}$ fits in the following exact sequence:
$$0\to \mathsf{H} \to \mathsf{G} \to \mathsf{A} \to 0.$$
We have a rational section $s$ from $\mathsf{A}(\mathbb{C})$ to $\mathsf{G}(\mathbb{C})$ defined by
$$[\wp(u):\wp'(u):1] \to (u, g(u)).$$
Then the extension $\mathsf{G}$ is determined by the symmetric rational system of factor of $\mathsf{A}(\mathbb{C})$ with values in $\mathsf{H}(\mathbb{C})$ given By
$$\varphi([\wp(x):\wp'(x):1],[\wp(y):\wp'(y):1])=f(x+y)\ominus f(x) \ominus f(y).$$
The quasi-periodic relations of Weierstrass zeta
and sigma functions show that the image of $\mathrm{Ext}(\mathsf{A}, \mathsf{H})$ in $\mathbb{R}$ (resp. $\mathbb{R}/\mathbb{Q}$ and $\sqrt{-1}\mathbb{R}/a\sqrt{-1}\mathbb{Q}$) under the canonical map constructed in \cite{Ser} is given by
$$\frac{a\sqrt{-1}\eta_{a\sqrt{-1}}(1)-\eta_{a\sqrt{-1}}(a\sqrt{-1})}{2a\sqrt{-1}}=\frac{\pi}{a} \neq 0~(\text{resp.}~\xi~\text{and}~\xi\sqrt{-1} )$$
in additive case (resp. in multiplicative case and in twisted multiplicative case). In additive case, the element $\frac{\pi}{a}$ is not zero, so the extension $\mathsf{G}$ of $\mathsf{G}_{a}$ by $\mathsf{A}$ is the unique non-split extension. In the multiplicative case (resp. twisted multiplicative case) the equivalence classes of the number $\xi$ (resp. $\xi\sqrt{-1}$) can be chosen as the parameter in our classification described in the section.

\end{document}